\documentclass[11pt]{amsart}

\usepackage{amssymb,amsmath,enumerate,epsfig}
\usepackage{amsfonts,color}
\usepackage{a4,latexsym, multirow}
\usepackage{parskip}
\usepackage{hyperref}
\hypersetup{pdftitle=Segre4}

\newcommand{\C} {\mathbb{C}}

\newcommand{\Z}{\mathbb{Z}}

\newcommand{\PP}{\mathbb{P}}
\newcommand{\NS}{\mathop{\rm NS}}

\newcommand{\mZ}{\mathcal Z}

\newcommand{\KK}{k}
\newcommand{\mF}{\mathcal F}
\newcommand{\mR}{\mathcal R}

\newtheorem{Theorem}{Theorem}[section]
\newtheorem{Proposition}[Theorem]{Proposition}
\newtheorem{Lemma}[Theorem]{Lemma}

\newtheorem{Corollary}[Theorem]{Corollary}
\newtheorem{Assumption}[Theorem]{Assumption}

\theoremstyle{remark}
\newtheorem{Remark}[Theorem]{Remark}

\newtheorem{Example}[Theorem]{Example}

\theoremstyle{definition}

\newtheorem{Definition}[Theorem]{Definition}

\begin{document}

\title{64 lines on smooth quartic surfaces}


\author{S\l awomir Rams}
\address{Institute of Mathematics, Jagiellonian University, 
ul. {\L}ojasiewicza 6,  30-348 Krak\'ow, Poland}
\address{Current address: 
Institut f\"ur Algebraische Geometrie, Leibniz Universit\"at
  Hannover, Welfengarten 1, 30167 Hannover, Germany} 
\email{slawomir.rams@uj.edu.pl}

\author{Matthias Sch\"utt}
\address{Institut f\"ur Algebraische Geometrie, Leibniz Universit\"at
  Hannover, Welfengarten 1, 30167 Hannover, Germany}
\email{schuett@math.uni-hannover.de}

\thanks{Funding   by ERC StG~279723 (SURFARI) and  NCN grant N N201 608040 (S.~Rams)
 is gratefully acknowledged.}

\date{May 19, 2015}

\begin{abstract}
Let $\KK$ be a field of characteristic $p\geq 0$ with $p\neq 2,3$.
 We prove 
 that there are no geometrically smooth quartic surfaces $S \subset \PP^3_k$ 
 with more than 64 lines. 
 As a key step, we derive the sharp bound
that any line meets at most 20 other lines on $S$. 
\end{abstract}
%
%
 \maketitle

 \section{Introduction}
 \label{s:intro}
 
 The aim of this paper is to 
study the configurations of lines on certain quartic surfaces.
In particular, we prove:

\begin{Proposition}
\label{prop1}
\begin{enumerate}
\item
A line $\ell$ on a geometrically smooth quartic surface $S$ in $\PP^3_k$
  intersects at most $20$ other lines provided that char$(k)\neq 2,3$.
  \item
  \label{b}
If $\ell$ meets more than 18 lines on $S$,
then $S$ can be given by a quartic polynomial
$$
x_3x_1^3+x_4x_2^3+x_1x_2 q(x_3,x_4)+x_3g(x_3,x_4)
$$
where $q, g\in \bar k[x_3,x_4]$ are homogeneous 
of degree 2 resp.~3.
\item
\label{c}
The line $\ell=\{x_3=x_4=0\}$ meets 20 lines on $S$ if and only if $x_4\mid g$.
\end{enumerate}
\end{Proposition}

The above proposition will enable us to  prove the main theorem of this paper:
 
\begin{Theorem}
\label{thm}
Let $\KK$ be a field of characteristic $p\geq 0$ with $p\neq 2,3$.
Then any geometrically smooth quartic surface over $k$ contains at most 64 lines.
\end{Theorem}

Historically this problem goes back to the study of 
lines on a smooth cubic surface in $\PP^3_\C$ 
 by Cayley  \cite{Cayley}, Salmon \cite{Salmon} and Clebsch.
Notably, on smooth cubics, number and configuration of lines are independent of the 
chosen surface.
 For higher degrees,
 the situation is completely different
 as a generic surface does not contain any line at all.
 Starting from first bounds due to Salmon and Clebsch 
(see Lemma \ref{lem:flecnodaldivisor}),
the problem of the maximal number of lines (for fixed degree)
continues to be an important topic of current research
(cf.~\cite{BS}, \cite{Miyaoka}).
In this direction,
a cornerstone is marked
 by Segre's 1943 paper \cite{Segre}
claiming 
that a line on a smooth complex quartic intersects at most $18$ other lines  \cite[p.~94]{Segre}.
%
This erroneous claim was the key point in Segre's attempt to show
 that a smooth complex quartic surface
may contain at most 64 lines (see Sect.~\ref{s:I}).
However, in view of Propostion \ref{prop1},
Segre's arguments can only yield a bound of 
72 lines on a smooth  quartic surface.

The proofs of  Proposition \ref{prop1} and Theorem \ref{thm}
make heavy use of the theory of elliptic fibrations
 which was not at Segre's disposal.
 In particular, we often apply the fact that smooth quartics in $\PP^3$,
 like any K3 surface, may admit several elliptic fibrations.
 The case of quartics with lines is specifically instructive
 since any line  induces an elliptic fibration (see Sect.~\ref{s:set-up}). 
Modern techniques enable us to (partially) control the behaviour of  its singular fibres  and
describe the quartics from Proposition \ref{prop1} 
that crucially violate Segre's claim. 
While Segre's strategy can no longer work for these surfaces,
it turns out that they have a 
rich structure, for instance reflected in a symplectic 
automorphism of order 3.
The detailed study of the geometry of these quartics 
in   Sect.~\ref{s:II} completes the proof of Theorem~\ref{thm}. 
We do not discuss here all the gaps and mistakes in certain arguments of \cite{Segre}, 
because our main goal is to prove Theorem~\ref{thm}.
Instead we defer   
the treatment of these questions, in a more general framework, to \cite{RS}.

 We should point out that Theorem \ref{thm} is sharp: 
 in any characteristic $p\neq 2, 3$,
 Schur's quartic (\cite{schur})
 \begin{eqnarray}
 \label{eq:Schur}
 \{x^4-xy^3=z^4-zw^3\}\subset\PP^3
 \end{eqnarray}
 contains exactly 64 lines.
 This can be verified with a method going back at least to Barth \cite{Barth}
 and detailed in \cite{BS}.
 In contrast, for characteristic $3$,
 Schur's quartic admits a model with good reduction
 which is isomorphic to the Fermat quartic and thus contains 112 lines. 
It is non-trivial to prove that 112 is the maximum number of lines 
on a smooth quartic surface in characteristic $3$
(because the flecnodal divisor can be degenerate, see Lemma 
\ref{lem:flecnodaldivisor} and \cite{RS14}).

In the same paper Segre claims to have obtained 
"a new non-singular quartic surface containing 64 lines" (\cite[p.~88]{Segre}), 
but we were unable to track down such a construction in his extensive oeuvre.
In the last section of the paper we  derive some properties of the configuration of $64$ lines on a smooth quartic (Proposition \ref{prop:=64}).
These properties form a strong evidence for the claim that a smooth quartic with maximum number of lines is unique
(up to isomorphism over an algebraically closed field).

 While this research was conducted, we were informed by A.~Degtyarev and I.~Itenberg 
 that they had found a  heavily computer-aided   proof of 
 Theorem~\ref{thm} over $\C$ using the lattice theory of K3 surfaces
 (a problem posed by Barth in  \cite[p.~33]{Barth}),
 which would also imply the uniqueness of a smooth complex quartic with maximal number of lines up to projective equivalence. 
In contrast, our treatment makes no essential use of computer-aided lattice calculations
and works for all characteristics other than $2,3$. 

The interested reader may also want to compare a recent paper  by Miyaoka 
\cite{Miyaoka}
which gives bounds for the number of lines on certain complex surfaces,
but 
does not apply to quartics in $\PP^3$ (cf.~\cite[p.~921]{Miyaoka}).

 \noindent
{\em Convention:} Unless otherwise indicated, in this note we work over an algebraically closed field $\KK$ of characteristic $p \neq 2,3$.
Obviously this will suffice to prove Theorem \ref{thm}.
 
 \section{Set-up}
 \label{s:set-up}

 Let $S$ be a smooth quartic surface over 
the field $k$
 containing a line $\ell$.
 Then the pencil of planes $H_t$ containing $\ell$
induces a fibration
\begin{equation} \label{eq:fibration}
\pi: S\to\PP^1
\end{equation}
with reduced fibres of genus 1. 
If char$(k)\neq 2,3$ (which we assume throughout this note), 
then the general fibre is automatically a smooth cubic
with 9 points of inflection.
By construction,
any line on $S$ meeting $\ell$ occurs as component of a singular fibre
$F_t$ of \eqref{eq:fibration}, $t\in\PP^1$.
Therefore it is important to understand the possible singular fibres;
these can be classified following Kodaira \cite{K}, N\'eron \cite{Neron} and Tate \cite{Tate}.
In Kodaira's notation there are the following types:
%

\begin{table}[ht!]
\begin{tabular}{c|l}
fibre type & configuration\\
\hline
$I_1$ & nodal cubic\\
$I_2$ & 2 rational curves meeting transversally in 2 points\\
$I_3$ & 3 rational curves meeting transversally in 3 points\\
$II$ & cuspidal cubic\\
$III$ & 2 rational curves meeting tangentially in a point\\
$IV$ & 3 rational curves meeting transversally in a point
\end{tabular}
\vspace{.3cm}
\label{tab:F}
\caption{Singular fibres of elliptic pencils on smooth quartic surfaces}
\end{table}

For the geometry of the quartic,
it turns out to be crucial how the line $\ell$ meets the fibres. 
Following \cite[p.~87]{Segre} we recall the following definition:

\begin{Definition}[Lines of the second kind]
The line  $\ell$ on $S$ is of the second kind 
iff it intersects every smooth fibre of the fibration
\eqref{eq:fibration} in a point of inflection. 
Otherwise,   $\ell$ is called a line of the first kind.
\end{Definition}

If we assume the line $\ell$ to be of the second kind, then $\ell$ meets
the singular fibres  in points lying in the closure of the flex locus of the smooth fibres
of \eqref{eq:fibration}.
Since the group structure carries over from the smooth fibres
to the smooth locus of the singular fibres by means of N\'eron's model \cite{Neron},
one can compute the support of the closure of the flex locus on the singular fibres
as given in Table \ref{tab:F3}.

\begin{table}[ht!]
\begin{tabular}{c|l}
fibre type & configuration\\
\hline
$I_1$ & 3 smooth points, the node\\
$I_2$ & 3 smooth points of the line component, both nodes\\
$I_3$ & 3 smooth points on each component\\
$II$ & 1 smooth point, the cusp\\
$III$ & 1 smooth point of the line component, the tacnode\\
$IV$ & 1 smooth point on each component,  the triple point
\end{tabular}
\vspace{.3cm}
\label{tab:F3}
\caption{Support of the closure of the flex locus on the singular fibres}
\end{table}

In the next two sections, we will study elliptic fibrations induced by lines of the second kind in detail.
We will show that these occur in three distinct families (Proposition \ref{prop}).
For two of the three families, as well as for quartics without any lines of the second kind,
there is a relatively simple proof of Theorem~\ref{thm},
which we work out in Section \ref{s:I}.
The third family, which shall be denoted by $\mZ$, however, requires serious additional work
which we carry out in Section \ref{s:II}.
Notably this family contains quartics that crucially contradict Segre's claims
from \cite{Segre}
(see also \cite{RS}).

\section{Base change and ramification}
\label{s:3}

Any line $\ell$ on a smooth quartic $S$ comes equipped not only with a genus 1 fibration $\pi: S\to \PP^1$,
but also with a natural morphism
\[
\ell\to\PP^1
\]
of degree 3, sending $x\in\ell$ to the unique point $t\in\PP^1$ such that $x\in F_t$.
On the level of function fields, this corresponds to a degree $3$ extension $k(\ell)/k(t)$.
Its Galois closure corresponds to the function field of another curve $B$;
here $k(B)$ has degree $1$ or $2$ over $k(\ell)$.
Overall, the extension $k(B)/k(t)$ has degree $n=3$ or $6$.
On the level of genus one fibrations, this corresponds to a sequence of base changes
$$
\begin{array}{ccccc}
S_2 & \to & S_1 & \to & S\\
\downarrow && \downarrow &&\downarrow\\
B &\to & \ell & \to & \PP^1
\end{array}
$$
By construction, the trisection $\ell$ of $\pi$ splits off a section on $S_1$
(so that $S_1\to \ell$ is a proper elliptic fibration,
i.e.~the generic fibre has a rational point);
on $S_2$, $\ell$ splits completely into sections.
Choose one of them as zero section $O$
and denote the other two sections by $P_1, P_2$.

We now specialise to the case where $\ell$ is a {\bf line of the second kind}
(for lines of the first kind, similar arguments work, see Section \ref{s:64}).
Then by construction $P_1, P_2$ are 3-torsion sections  which are inverse to each other.
In order to work out the possible singular fibres of $\pi$,
we argue with the fibration 
\[
S_2\to B \, , 
\]
since there the sections severely limit the number of cases.
Notably the 3-torsion sections meet any fibre (smooth or non-smooth)
in three distinct smooth points.
By N\'eron's classification \cite{Neron}, this leaves only the following fibre types:
\begin{eqnarray}
\label{eq:F3}
I_n (n\geq 0), IV, IV^*.
\end{eqnarray}
We shall now distinguish the possible fibre types 
according to the ramification of the morphism $\ell\to\PP^1$.
Note that this morphism is automatically unramified in the smooth fibres
and separable outside characteristic $3$ by \cite[Ch.~7, Theorem 4.38]{Liu}.

\begin{Lemma}[Unramified fibres]
\label{lem:unram}
Let $F$ be a singular fibre of $\pi$
such that the map $\ell\to\PP^1$ is unramified at $F$.
Then $F$ has type $I_1, I_3$ or $IV$.
\end{Lemma}

\begin{proof}
Since the fibre is unramified, 
$\ell$ meets $F$ in 3 distinct points, all necessarily smooth.
On $S_2$, the fibre $F$ is replaced by $n$ fibres of the same type.
Since these fibres accommodate non-trivial 3-torsion sections by construction,
they fall into the types occurring in Table \ref{tab:F3}.
A priori this leaves the fibre types $I_1, I_2, I_3, IV$.
However,
$I_2$ is impossible since the 3-torsion sections would necessarily all meet the same fibre component
while the line $\ell$ inducing them meets both fibre components as all three curves
(line $\ell$, residual line and conic) lie in the same $H_t\cong\PP^2$.
\end{proof}

Next we analyze the ramified fibres
(which are automatically singular).
At such a fibre, the morphism $\ell\to\PP^1$ has either 2 pre-images (ramification type $1$),
or it is cyclic of degree $3$ (ramification type $2$).
A similar analysis has been carried out in \cite[\S 4.2, case b)]{HT},
but it seems that some cases have gone missing there.

\begin{Lemma}[Ramified fibres]
\label{lem:ram}
Let $F$ be a  fibre of $\pi$
such that the map $\ell\to\PP^1$ is ramified at $F$.
Then $F$ has type $I_1, I_2, II$ or $IV$ according to the  ramification type as follows:

\begin{table}[ht!]
\begin{tabular}{c|cc}
fibre type & $II$ & $I_1, I_2, IV$\\
ramification type & $1$ & $2$
\end{tabular}
\end{table}
\end{Lemma}

\begin{proof}
First of all, $F$ cannot have type $I_3$ since this fibre admits full 3-torsion
so that $\ell\to\PP^1$ cannot possibly ramify.
Similarly, if $\ell\to\PP^1$ ramifies at a fibre of type $IV$, then $\ell$ meets the node of the fibre,
and the ramification type is $2$.
The same argument applies to fibre type $I_2$ by Table \ref{tab:F3},
and in fact, it rules out type $III$ right away.
For the other fibre types, we argue with the base change 
\[
S_2\to B.
\]
In $S_2$, the fibre $F$ is replaced by fibres admitting 3-torsion, i.e.~from Table \ref{tab:F3}.
For instance, a fibre of type $IV$ is replaced by a smooth fibre.
This can be deduced from Tate's algorithm \cite{Tate}
which generally encodes the behaviour of singular fibres 
under cyclic base change of degree $d$.
As an example, we justify the above claim about type $IV$ fibres.
Locally (outside characteristics $2,3$), Tate's algorithm shows that they can be given by a Weierstrass form
\[
y^2 = x^3 + t^2 Ax + t^2 B
\]
with local parameter $t\nmid B$.
Quadratic base change $t=s^2$ directly leads to type $IV^*$
while cyclic base change $t=s^3$ result in a non-minimal Weierstrass form. 
Upon minimalising by $(\xi,\eta)=(s^2x,s^3y)$,
we obtain 
\[
\eta^2 = \xi^3 + s^2 A(s^3)\xi + B(s^3)
\]
with a smooth fibre of j-invariant zero at $s=0$.
All other types can be treated along similar lines.
Below we only list the cases relevant to our problem:
$$
\begin{array}{ccc}
d=1 & d=2 & d=3\\
\hline
I_n & I_{2n} & I_{3n}\\
II & IV & I_0^*\\
III & I_0^* & III^*\\
IV & IV^* & I_0
\end{array}
$$

We note that the above behaviour would have sufficed to rule out the fibre type $III$, 
since in our setting it cannot possibly be base-changed to fibres from  \eqref{eq:F3}.
For $F$ of type $II$, the only possibility consists in replacing it by 3 fibres of type $IV$ in $S_2$.
On $\ell\to\PP^1$, this corresponds to the non-Galois case of ramification type $1$.

It remains to discuss  fibres of type $I_1$.
Assume that the ramification type is $1$.
That is, $\ell$ meets a smooth point of the fibre transversally and a node. 
In $S_2$ the fibre would be replaced  by fibres of type
$I_2$.
By assumption, the section $O$ would meet one fibre component on $S_2$ 
(corresponding to the original smooth intersection point)
while the other sections $P_1, P_2$ meet the other component
(corresponding to the node).
But the structure of the fibre components is $\Z/2\Z$.
Since this does not accommodate proper 3-torsion,
all 3-torsion section have to meet the same component, contradiction.
\end{proof}

As for the existence of the above-listed fibre types, one can exhibit explicit smooth quartics in $\PP^3$
with a line of the second kind and given ramification type without too much difficulty.
For a particular configuration, we will do this in Lemma \ref{lem:familygeil33}.

\section{Consequences for lines of the second kind}
\label{s:cons}
\label{s:4}

In \cite[$\S$~7]{Segre},  Segre claims that a line $\ell$ of the second kind on a complex quartic surface
meets at most 18 other lines.
Here we confirm his claim for two configurations of ramification types of $\ell\to\PP^1$.
For the third configuration, however, we give  a counterexample
and correct Segre's bound.
We also correct Segre's assertion \cite[p.~95]{Segre} that for the second configuration, the number of lines meeting $\ell$
always equals 15.
Throughout, our arguments work over any algebraically closed field $k$ of characteristic $\neq 2,3$.

Note that since $\ell\to\PP^1$ has degree $3$,
the ramification types admit three configurations:
\begin{itemize}
\item
4 points with 2 pre-images and ramification of index $1$ each, or
\item
2 points with 2 pre-images each and 1 point with 1 pre-image (ramification of index $1,1,2$), or
\item 
2 points with 1 pre-image each (ramification of index $2$ each).
\end{itemize}
We will use the following shorthand-notation for the ramification type $R$:
\[
R=1^4;\text{ resp. } R =  2, 1^2;\text{ resp. } R = 2^2.
\]

\begin{Proposition}[Lines meeting $\ell$]
\label{prop}
Let $\ell$ be a line of the second kind on a smooth quartic $S$ with ramification type $R$.
Let $G_R$ be defined as follows:

\begin{table}[ht!]
\begin{tabular}{c|ccc}
$R$ & $1^4$ & $2,1^2$ & $2^2$\\
\hline
$G_R$ & $\{12\}$ & $\{15,16\}$ & $\{18,19,20\}$
\end{tabular}
\end{table}

Then $\ell$ meets exactly $N$ other lines contained in $S$, where $N\in G_R$.
\end{Proposition}

\begin{Remark}
\label{rem:4.2}
Regarded as data concerning degree $3$ morphisms $\PP^1\to\PP^1$,
ramification type $R=2^2$ is most special, and in fact arises as a degeneration of the other two.
In our context of lines of the second kind, however,
somewhat counterintuitively,
each ramification type comes with its own 6-dimensional family of quartics.
Here we will only exhibit the family for type $R=2^2$ in Lemma \ref{lem:Z}
which is crucial to the proof of Theorem \ref{thm}.
An analysis of the other families,
each interesting in its own right, is directed to \cite{RS}.
\end{Remark}

To prove the Proposition,
we use the $3$-torsion sections on $S_2$.
These induce an isogeny (on the generic fibres)
\[
S_2 \to S_2'.
\]
Here $S_2$ and $S_2'$ have the same invariants (such as Euler-Poincar\'e characteristic, Betti numbers, geometric genus, Picard number).
As has been exploited before (see e.g.~\cite{MP}),
this puts severe restrictions on the singular fibres of $S_2$ and $S_2'$.
Recalling the possible singular fibres of $S$ from Lemmas \ref{lem:unram} and \ref{lem:ram},
we can easily determine the corresponding fibres on $S_2$ and $S_2'$
in terms of the degree $n$ of the Galois extension $k(B)/k(t)$:

\begin{table}[ht!]
\begin{tabular}{c|ccc}
& fibre on $S$ & fibre on $S_2$ & fibre on $S_2'$\\
\hline
& $I_1$ & $n\times I_1 $ & $n\times I_3$\\
unramified & $I_3$ & $n\times I_3 $ & $n\times I_1$\\
& $IV$ & $n\times IV $ & $n\times IV$\\
\hline
& $II$ & $3\times IV$ & $3\times IV$\\
ramified & $I_1$ & $I_3$ & $I_9$\\
& $I_2$ & $I_6$ & $I_{18}$\\
& $IV$ & $I_0$ & $I_0$
\end{tabular}
\end{table}

To see this in the unramified case, 
it suffices to note that $\ell$ meets all three components of
a fibre of type $I_3$, so the three sections on $S_2$ meet different components of each pre-image.
For the ramified semi-stable case,
recall that $\ell$ meets the fibre at a node with multiplicity 3,
i.e. tangentially to one of the branches.
This implies that all three sections on $S_2$ meet the same component
of the resulting singular fibre,
and the fibre type on $S_2'$ is as claimed.

Since the Euler-Poincar\'e characteristics on $S_2$ and $S_2'$ are the same, and since they equal the sum of the local contributions from the singular fibres,
we deduce that the unramified fibres of type $I_3$ on $S_2$ have to balance out with the other semi-stable fibres.
On $S$, this translates as follows: 

\begin{Lemma}
\label{lem:fibres}
Semi-stable fibres on $S$ occur in pairs $(I_1, I_3)$ and triples $(I_2, I_3, I_3)$.
\end{Lemma}

We shall now prove Proposition \ref{prop} by a case-by-case analysis 
depending on the ramification type.

$\boldsymbol{R=1^4}$.
By Lemma \ref{lem:ram}, $S$ has 4 ramified fibres of type $II$.
Since these contribute $8$ to $e(S)=24$,
the remaining local contributions of $16$ disperse over a total number of 4 
pairs of $(I_1, I_3)$ and fibres of type $IV$.
In any case, the fibration $S\to\PP^1$ has exactly 4 fibres consisting of 3 lines,
and no other reducible fibres. Thus $\ell$ meets exactly 12 lines $\ell' \neq \ell$ in $S$. 

$\boldsymbol{R=2,1^2}$.
By Lemma \ref{lem:ram}, $S$ has 2 ramified fibres of type $II$.
Since these contribute $4$ to $e(S)=24$,
the remaining local contributions of $20$ disperse over a total number of 5 
pairs of $(I_1, I_3)$ and fibres of type $IV$, or possibly (counted twice) a triple $(I_2,I_3,I_3)$.
In any case, the fibration $S\to\PP^1$ has exactly 5 fibres consisting of 3 lines,
and possibly one further reducible fibre
(line + conic). Thus $\ell$ meets exactly 15 or 16 lines $\ell' \neq \ell$ in $S$. 

$\boldsymbol{R=2^2}$.
By Lemma \ref{lem:ram}, $S$ has no fibres of type $II$.
In order to reach $e(S)=24$,
the local contributions of the singular fibres are distributed over a total number of 6 
pairs of $(I_1, I_3)$ and fibres of type $IV$, or possibly (counted twice each) one or two triples $(I_2,I_3,I_3)$.
In any case, the fibration $S\to\PP^1$ has exactly 6 fibres consisting of 3 lines,
and possibly one or two further reducible fibres
(line + conic). Thus $\ell$ meets exactly 18, 19 or 20 lines $\ell' \neq \ell$ in $S$.
This completes the proof of Proposition \ref{prop}. \qed

\begin{Remark}
Compared with the analysis in \cite[\S 4.2 case b)]{HT},
we have added a few cases, but the main statement of loc.~cit.~is still valid:
a line of the second kind meets at least 6 (or in fact 12) lines  on a smooth quartic.
\end{Remark}

The existence of each of the above cases can be proved by exhibiting
explicit smooth quartics.
In fact, each configuration can be realised in a different 6-parameter family of quartics.
In particular, this shows that neither occurs as a subfamily of the other 
(cf.~Remark \ref{rem:4.2}).
Segre's original argument in \cite{Segre} is only valid for the first configuration, i.e.~the case $R=1^4$.
We will elaborate on this in \cite{RS}.
Here we give the explicit parametrisation of the case $R=2^2$.
In particular this answers a question from \cite[\S 4.2 case b)]{HT}
(see also Lemma \ref{lem:degen} for degenerations).

\begin{Lemma}[Family $\mZ$]
\label{lem:familygeil33}
\label{lem:Z}
Let $\ell$ be a line of the second kind on a smooth quartic $S$ with ramification type $R=2^2$.
Then $S$ is projectively equivalent to a quartic in the family
$$
\mZ := \{ x_3x_1^3+x_4x_2^3+x_1x_2q(x_3,x_4)+g(x_3,x_4)=0\} \, , 
$$
where $q  \in k[x_3, x_4]$ (resp. $g \in k[x_3, x_4]$) is  a homogeneous polynomial of degree $2$ (resp.~$4$).
\end{Lemma}

\begin{proof}
After a linear transformation, we can assume the line $\ell$ to be given by $x_3=x_4=0$,
and the ramification to occur at $0$ and $\infty$, that is at $x_3=0$ and $x_4=0$.
A further normalisation takes the residual cubic powers in these fibres to $x_1^3, x_2^3$,
so that the equation becomes
\begin{eqnarray}
\label{eq:mid}
x_3x_1^3+x_4x_2^3+x_3^2q_1+x_3x_4q_2+x_4^2q_3=0
\end{eqnarray}
where the $q_j$ are homogeneous quadratic forms in $x_1,\hdots,x_4$.
By adding terms in $x_3, x_4$ to $x_1$ and $x_2$, 
we can make sure that the terms involving $x_1^2x_3$ and $x_2^2x_4$ in \eqref{eq:mid} are zero.
Then solving for $\ell$ to be a line of the second kind,
i.e.~for the Hessian of \eqref{eq:mid} to vanish identically on $\ell$,
implies all other terms not appearing in $\mZ$ to vanish, too.
\end{proof}

We end this section by commenting on the elliptic fibration
\[
\pi: S \to \PP^1
\]
induced by the line $\ell$ of the second kind on a smooth K3 surface $S \in \mZ$.
Generically, there are six singular fibres of Kodaira type $I_1$
located at $0, \infty$ and at the zeroes of $g$.
Using standard formulas, for instance for the Jacobian of the fibration,
one finds 6 fibres of Kodaira type $I_3$ at the zeroes of
$q^3+27x_3x_4g$.
Thus we can also describe the degenerations of $S$ as smooth quartic
which we will need for the proof of Proposition \ref{prop1} (b), (c):

\begin{Lemma}[Degenerations]
\label{lem:degen}
A surface $S\in\mZ$ is a smooth quartic
such that the fibration $\pi:S\to\PP^1$ attains a fibre of Kodaira type $I_2$ (necessarily at $0$ or $\infty$) iff $x_3$ or $x_4$ divides $g$.
The ramified fibres degenerate to Kodaira type $IV$ iff $x_3$ or $x_4$ divides $q$.
\end{Lemma}

\begin{proof}
The claimed degenerations are obvious from the loci of the singular fibres.
It only remains to note that a singular fibre of Kodaira type $I_2$ outside $0,\infty$,
requires $g$ to contain a square factor.
But then this results in a singularity of $S$ at the double root with $x_1=x_2=0$.
\end{proof}

\section{64 lines away from $\mZ$}
\label{s:I}

This section adapts Segre's arguments from \cite{Segre} 
to prove Theorem~\ref{thm} for quartics $S$ outside the family $\mZ$,
i.e.~if there is no line of the second kind on $S$, or all lines of the second kind on the quartic in question fall into the first two cases of 
Proposition~\ref{prop}.
We make crucial use of the flecnodal divisor on $S$  (see \cite{Clebsch}, \cite{Salmon}).
Classically this divisor was studied over $\C$, but by \cite[Thm~1 \& Prop.~1]{Voloch}, 
it is also non-degenerate in
almost all positive characteristics (generally when $p$ exceeds the degree of the surface in question):

\begin{Lemma}[Flecnodal divisor]
\label{lem:flecnodaldivisor}
Let $p\neq 2,3$ and let $S \subset \PP^3(\KK)$ be a smooth quartic surface. 
Then there exists an effective divisor
$\mF_S \in {\mathcal O}_S(20)$ such that 
\[
\operatorname{supp}(\mF_S) = \{P\in S; \text{ there exists  a line } L \subset\PP^3 \mbox{ with }  i_P(S,L)\geq 4\}.
\]
\end{Lemma}

Note that any line contained in $S$ is  a component of the support of $\mF_S$.
Thus we get an automatic upper bound of 80 lines on any smooth quartic in characteristic different from $2, 3$ (see \cite[p.~106]{Clebsch}, \cite[Cor.~1]{Voloch}).
Similarly, for a curve  $C$ of degree $c$ on $S$ not contained in the support of $\mF_S$,
we have $C.\mF_S = 20c$ and thus $C$ meets at most $20c$ lines on $S$.
Often this can be improved by studying the divisor obtained from $\mF_S$ by subtracting all lines contained in $S$,
or by studying the intersection behaviour with some specific lines in detail, 
as we shall exploit in this and the next section.

Recall that any line (just like any smooth rational curve)
has self-intersection number $-2$ on $S$.
For a line $\ell\subset S$, we obtain
$\ell.(\mF_S-\ell)=22$,
so $\ell$ meets at most 22 other lines on $S$, unless $\ell$ has multiplicity strictly greater than 1 in $\mF_S$.
For lines of the second kind, we have seen better bounds in Proposition \ref{prop}.
For lines of the first kind, there is a different elementary argument due to Segre improving the above bound:

\begin{Lemma} {\rm (\cite[p.~88]{Segre})}
\label{lem:linesfirstkind}
\label{lem:I}
Let $\ell$ be a line of the first kind on a smooth quartic $S$.
Then $\ell$ is met by at most $18$ lines  $\ell' \neq \ell$ on $S$.
\end{Lemma}

\begin{proof}  
The proof for complex quartics can be found in \cite[$\S$~2]{Segre}. 
Since the arguments are genuinely geometric,
using points of inflection and the Hessian,
they remain valid in the more general set-up under consideration
as we sketch below for completeness.

After a projective transformation
we may assume that $\ell = \operatorname{V}(x_3, x_4)$.
Fix a generator of the ideal of $S$,
$$
f = \sum_{i,j=1}^{4} x_3^i \cdot x_4^j \cdot \alpha_{i,j}(x_1, x_2)\;\; \mbox{ where } \alpha_{i,j} \in  \KK[x_1,x_2].
$$
 If necessary, a linear transformation in $x_3, x_4$ ensures 
 that the cubic residual to the line $\ell$ in  $S \cap \operatorname{V}(x_3)$
is smooth and the line $\ell$ contains none of its inflection points.
We put $\Gamma_{\lambda}$ to denote the planar cubic  over $\KK(\lambda)$
residual to $\ell$
in the intersection $S \cap \operatorname{V}(x_4 - \lambda x_3)$.
The curve $\Gamma_{\lambda}$ meets the line  $\ell$ in three points given by 
the polynomial
\begin{equation*} \label{eq-restrictionofthecubic} 
g_{\lambda}(x_1,x_2) := \alpha_{1,0}(x_1, x_2) + \lambda  \cdot \alpha_{0,1}(x_1, x_2).
\end{equation*}
We shall be interested in the problem when these include points of inflection of  $\Gamma_{\lambda}$.
Equivalently, the  determinant of the Hessian of  $\Gamma_{\lambda}$  vanishes when restricted to $\ell$:  
\begin{equation*} \label{eq-hessian}
h_{\lambda}(x_1,x_2) := \operatorname{det} 
\left[\begin{array}{ccc}
\frac{\partial^2 g_{\lambda}}{\partial x_1^2}                &  \frac{\partial^2 g_{\lambda}}{\partial x_1 \partial x_2} &  
\frac{\partial \mbox{ \vspace*{2ex}} }{\partial x_1} \sum_{i+j=2} \lambda^j \alpha_{i,j} \\
\vdots    & \frac{\partial^2 g_{\lambda}}{\partial x_2^2}             & 
\frac{\partial \mbox{ \vspace*{2ex}} }{\partial x_2} \sum_{i+j=2} \lambda^j \alpha_{i,j}   \\
                                                           & \ldots  & \sum_{i+j=3} \lambda^j \alpha_{i,j}                        
\end{array}
\right] = 0 \,.
\end{equation*} 
For generic $\lambda$,  the homogeneous polynomials $g_{\lambda}, h_{\lambda}$ 
have no common factor
because  the line $\ell$ is of the first kind.
 Thus their resultant with respect to $x_1$
gives a non-zero polynomial $r(\lambda)\in\KK[\lambda]$ times a power of $x_2$.
Since $r(\lambda)$ is the determinant of the Sylvester matrix of 
$g_{\lambda}(x_1, 1) \,, h_{\lambda}(x_1, 1)$, 
we have
\begin{eqnarray}
\label{eq:deg18}
\deg(r({\lambda})) \leq 18 \, .
\end{eqnarray}
To see this, one checks that the Sylvester matrix 
has three rows with entries of degree 5 in $\lambda$
 and three rows with entries of degree 1.
Thus, $\ell$ meets the residual cubics $\Gamma_\lambda$ in at most 18 points of inflection.
This count includes the singular fibres;
in fact,
if $\ell$ meets a line $\ell_0$ on $S$,
then the latter is a component of a residual cubic $\Gamma_{\lambda_0}$,
and the Hessian of $\Gamma_{\lambda_0}$ has zero determinant on $\ell_0$.
That is, $\ell_0$ contributes a zero at $\lambda_0$ to the polynomial $r(\lambda)$.
To deduce the claim of the lemma from \eqref{eq:deg18},
it remains to check that  if $\Gamma_{\lambda_0}$ decomposes into three lines, 
then
$\lambda_0$ is a triple root of $r(\lambda)$.
This can be verified by a direct (computer-aided) calculation.
%
\end{proof}


Combining Lemma \ref{lem:I} and Proposition \ref{prop},
we obtain the following general sharp bound 
(thanks to Lemma \ref{lem:degen}, Example \ref{ex}),
correcting Segre's original claim \cite[p.~94]{Segre}:

\begin{Corollary}[Proposition \ref{prop1}]
\label{cor:20}
Any line on a smooth quartic surface $S$ meets at most 20 other lines on $S$.
\end{Corollary}


\subsection{Proof of Proposition \ref{prop1}}

Part (a) is the content of Corollary \ref{cor:20}.
For (b), (c), assume that a line $\ell\subset S$ meets at least 19 other lines on $S$.
Then $\ell$ is of the second kind by Lemma \ref{lem:linesfirstkind}.
By Proposition \ref{prop}, the ramification type is $R=2^2$.
Hence $S$ is projectively equivalent (over $\bar k$)
to a quartic in the family $\mZ$ by Lemma \ref{lem:Z}.
Proposition \ref{prop1} (b), (c) 
then follow from Lemma \ref{lem:degen}.
\qed

\subsection{Proof of Theorem \ref{thm} away from the family $\mZ$}

An elliptic fibration  on a K3 surface can have at most 12 reducible fibres, of Kodaira type $I_2$, 
or less if Kodaira types $I_3, III, IV$ are involved. 
This follows from computing the Euler-Poincar\'e characteristic $e(S)=24$ 
as a sum of the fibre contributions
 (Segre speaks of the Zeuthen-Segre invariant in \cite[$\S$~8]{Segre}).

\begin{Lemma}[Twelve lines]
\label{lem:twelvelines}
If a line on $S$ is met by more than other $12$ lines, then it is met by $3$ coplanar lines.
\end{Lemma}

\begin{proof}
Each line meeting a given line $\ell$ is part of a reducible fibre of $\pi$.
If more than 12 lines would lie in fibres of Kodaira types $I_2, III$,
then their contributions to $e(S)$ would exceed $24$.
Hence, by inspection of  Table \ref{tab:F},
there are lines contained in a fibre of type $I_3$ or $IV$.
That is, they are coplanar with $\ell$.
\end{proof}

%
%
%
%
%
%
%

Now we can adapt the reasoning of  \cite[$\S$~9]{Segre} 
to show the following lemma.
\begin{Lemma}[64 lines]
\label{lem:sixtyfourlines}
\label{lem:64}
Assume that  the quartic $S$ is not projectively equivalent to a quartic from the family $\mZ$.

\begin{enumerate}
\item[(a)]
The surface $S$ contains at most $64$ lines.

\item[(b)] 
If there are $64$ lines on $S$, then four of them are coplanar lines of the first kind. 
\end{enumerate}
\end{Lemma}

\begin{proof}
{\it (a)} If $S$ contains $4$ coplanar lines, then each line on $S$ meets one of them. By Lemma~\ref{lem:linesfirstkind} and 
Proposition~\ref{prop} we obtain the desired bound.
Hence we can assume that $S$ contains no 4 (or 3) coplanar lines.

Suppose  that $S$ contains  two coplanar lines $\ell_1$, $\ell_2$ and the smooth conic in  
$|{\mathcal O}_S(1) - \ell_1 -\ell_2|$ is no  component of the flecnodal divisor.
By Lemma~\ref{lem:twelvelines} our assumption implies that  $\ell_1$, $\ell_2$ are met by at most $12$ lines on $S$.
By  Lemma~\ref{lem:flecnodaldivisor} the intersection number of the conic with  ${\mathcal F}_S$ is $40$,
so the conic is met by at most $40$ lines ($\ell_1$, $\ell_2$ included twice each). Altogether, we get at most 
$2+ 2 \cdot 11 + 36 = 60$ lines.

It remains to consider the case where for any two incident lines the residual conic is a component of  the flecnodal divisor.
If $S$ contains at least  $9$ such pairs of incident lines, 
then the number of lines is bounded by $\deg({\mathcal F}_S) - 2\cdot 9 = 62$.

Otherwise, we can assume that  $S$ contains only $N\leq 8$ pairs of incident lines. 
Thus any 
line on $S$ that is not an element of one of these pairs meets no lines on $S$ at all. 
Observe that the pairs of incident lines span a sublattice of $\NS(S)$ of rank at least $N+1$.
Since the Picard number of $S$ cannot exceed $22$,
 we infer that the number of lines on $S$ cannot exceed $(2 \cdot N) + (22 - (N+1)) = 21+N\leq 29$.

{\it (b)} The case-by-case analysis in the proof of (a)
shows that 64 lines require the existence of 4 coplanar lines on $S$
such that each in fact meets exactly 18 lines (in agreement with Lemma~\ref{lem:familygeil33}). 
Prop.~\ref{prop}
then rules out lines of the second kind, since the hypothesis of Lemma \ref{lem:64} avoids the third alternative of Prop.~\ref{prop} by Lemma \ref{lem:familygeil33}.
\end{proof}

\section{64 lines on $\mZ$}
\label{s:II}

The aim of this section is to prove the following proposition:

\begin{Proposition}[64 lines again]
\label{prop:64}
If the smooth quartic $S$ is projectively equivalent to a quartic from the family $\mZ$,
then it contains at most $64$ lines.
\end{Proposition}

For this purpose, we fix the smooth quartic $S\in\mZ$ and the line $\ell_0$ of the second kind 
with induced elliptic fibration 
\[
\pi_0: S\to \PP^1
\]
as described in Section~\ref{s:cons}.
We first prove the following lemma
which will help us narrow down the possible counterexamples to Proposition \ref{prop:64}
significantly.

\begin{Lemma}
\label{lem:I-II}
If a line $\ell'$ in a fibre of $\pi_0$ is of the second kind,
then 
either $\ell'$ meets at most 16 lines on $S$ (and the fibre is unramified, of Kodaira  type $I_3$)
or $S$ is the Schur quartic \eqref{eq:Schur}
(and the fibre is ramified, of Kodaira type $IV$).
\end{Lemma}

\begin{proof}
We begin by assuming that the fibre of $\pi_0$ containing $\ell'$ is ramified.
If the Kodaira type were to be $I_2$, then $\ell'$ induces a genus 1 fibration $\pi'$ with corresponding fibre of type $III$ (see below). This is impossible for lines of the second kind by 
Lemmata \ref{lem:unram}, \ref{lem:ram}.
Hence the fiber has Kodaira type $IV$,
and spelling out the conditions for $\ell'$ to be of the second kind leads directly to the Schur quartic.

Assume now that the fibre containing $\ell'$ is unramified.
If it were to have Kodaira type $IV$,
then on the genus 1 fibration $\pi'$  induced by $\ell'$ this would give a ramified $I_3$ fibre.
Thus $\ell'$ would be of the first kind by Lemma \ref{lem:ram}.
We conclude that the fibre containing $\ell'$ has Kodaira type $I_3$.
If $\ell'$ meets more than 16 lines on $S$,
then is has ramification type $R=2^2$ by Proposition \ref{prop}.
Spelling out the condition that $\pi'$ has only two ramified fibres,
one infers that the type $I_3$ fibre degenerates,
leading either to type $IV$ (which we ruled out above)
or to a singularity on $S$.
Hence $\ell'$ meets at most 16 lines on $S$.
\end{proof}

In other words, unless $S$ is projectively equivalent to the Schur quartic (with 64 lines),
any line of the second kind meeting more than 16 lines on $S$ is perpendicular
to all other lines of the second kind.
For the rest of this section, we make the following assumption
which we want to lead to a contradiction:

\begin{Assumption}
\label{ass}
$S$ contains more than 64 lines.
\end{Assumption}

We claim that the assumption implies that all lines in the fibres of $\pi_0$ are of the first kind.
Indeed,
since the Schur quartic contains exactly 64 lines
and thus is ruled out by Assumption \ref{ass},
we deduce from Lemma \ref{lem:I-II} that a line $\ell'$ 
of the second kind contained in a fibre of $\pi_0$ would meet at most 16 lines on $S$.
Now we can argue with a hyperplane $H$ decomposing on $S$ 
into $\ell_0$ and the fibre of Kodaira type $I_3$ or $IV$
comprising $\ell'$ and two other lines.
Each of these meets at most 18 lines on $S$ by Lemma \ref{lem:linesfirstkind}
and Lemma \ref{lem:I-II}.
Any other line on $S$ meets $H$ and thus exactly one of the four lines
that $H$ decomposes into. 
As required, we thus find
\[
\#\{\text{lines on } S\} \leq 4 + (16-3) + 2\cdot(18-3) + (20-3) = 64.
\]

We have just verified that all lines in the fibres of $\pi_0$ are of the first kind.
The completely analogous argument with a fibre of Kodaira type $I_3$ or $IV$
decomposing into lines of the first kind gives
\[
\#\{\text{lines on } S\} \leq 4 + 3\cdot(18-3) + (20-3) = 66
\]

\begin{Corollary}[66 lines]
\label{cor:66}
$S$ contains at most 66 lines.
\end{Corollary}

In particular, Assumption \ref{ass} requires that 
$\pi_0$ has $3\times 15=45$ lines as sections
and $\ell_0$
(and in fact any line of the second kind contained in $S$) meets more than the 18 generic lines on $S$.
Thus, by Lemma \ref{lem:degen}, $\pi_0$ admits  a (ramified) fibre of Kodaira type $I_2$.
We denote the fibre components by 

\begin{tabular}{cl}
$\ell_1$ & a line of the first kind by Lemma \ref{lem:I-II},\\
$Q$ & a conic which, by direct computation, is no component of $\mbox{supp}(\mF_S)$.
\end{tabular}

\subsection{Second elliptic fibration}
\label{ss:2nd}

As before, the line $\ell_1$ induces an elliptic fibration
\[
\pi_1 : S \to \PP^1.
\]
Since $\ell_0$ meets $Q$ tangentially in one of the two intersection points with $\ell_1$ 
(cf.~the analysis in the proof of Lemma \ref{lem:ram}),
 $\ell_0+Q$ gives a singular fibre $F_0$ of $\pi_1$ of Kodaira type $III$
 (which independently proves that $\ell_1$ is of the first kind by Lemmata \ref{lem:unram}, \ref{lem:ram}).
It is now about time to exploit the fact that the surface $S$ admits the symplectic automorphism
of order 3 
$$
\begin{array}{cccc}
\sigma: & S & \to & S\\
& [x_1,x_2,x_3,x_4] & \mapsto & [\varrho x_1,\varrho^2 x_2,x_3,x_4]
\end{array}
$$
where $\varrho$ is a primitive third root of unity.
The quotient $S/\sigma$ has a minimal resolution which is again a K3 surface $S'$.
Since the lines $\ell_0, \ell_1$ are fixed by $\sigma$,
$S'$ is equipped with elliptic fibrations induced by $\pi_0, \pi_1$.
In terms of  $\pi_0$, the quotient map $S\dasharrow S'$ is given by a 3-isogeny on the fibres
(visible immediately as soon as there is an additional line serving as a section);
notably $\sigma$ rotates the components of the $I_3$ and $IV$ fibres.
We also deduce for any line $\ell$ serving as a section of $\pi_0$  that 
$\ell^\sigma\neq \ell$ since $\ell^\sigma$ meets different components of the $I_3$ and $IV$ fibres.
Since $\ell_1$ is fixed by $\sigma$ as a set (with two fixed points 
where $\pi_1$ has the fibre of Kodaira type $III$ comprising $\ell_0$ and $Q$ 
as well as a smooth elliptic curve $E_0$ as fibre),
we find:


\begin{Lemma}
\label{lem:triples}
The lines met by $\ell_1$ other than $\ell_0$,
and the reducible singular fibres of $\pi_1$ other than $F_0$ come in triples.
\end{Lemma}

Since $\ell_1$ is of the first kind, it meets at most 18 lines on $S$ by Lemma \ref{lem:linesfirstkind}.
But here except for $\ell_0$, these lines come in triples by the above lemma, 
i.e.
\[
 \#\{\text{lines $\neq \ell_0$ on $S$ meeting }\ell_1\}=3N.
\]
Hence we deduce:

\begin{Corollary}
\label{cor:16}
The line $\ell_1$ meets at most 16 other lines on $S$ (i.e.~$N\leq 5$).
\end{Corollary}

Recall our assumption that $S$ contains more than 64 lines.
Any of these lines has to meet either $\ell_0, \ell_1$ or $Q$.
For the first and the last, the number of lines other than $\ell_0, \ell_1$ is bounded by
$19$ (Proposition \ref{prop}) and $36$ (since $Q$ is not contained in the flecnodal divisor,
but meets both $\ell_0, \ell_1$ with multiplicity 2).
Thus we find
\[
\#\{\text{lines on }S\} \leq 57 +3N\]
and therefore
\[
\#\{\text{lines on }S\}\geq 64 \Longrightarrow N\geq 3.
\]

\begin{Lemma}
\label{lem:config}
If $N\geq 3$, then the reducible singular fibres of $\pi_1$ have Kodaira types $III$ (consisting of $\ell_0, Q$) and  the following:
\begin{enumerate}
\item
\label{case1}
($N=5$)
$ 3 I_3, 6 I_2$;
\item
($N=4$)
$ 3 I_3$ or $IV$, $3 I_2$ or $III$;
\item
\label{case3}
($N=3$)
$3 I_3$ or $IV$;
\item
\label{case4}
($N=3$)
$9 I_2$ or $6 I_2, 3 III$.
\end{enumerate}
\end{Lemma}

\begin{proof}
The given configurations are compatible with the Euler-Poincar\'e characteristic $e(S)=24$
which is computed as the sum over the contributions from the (singular) fibres.
Any other configuration with given number of $3N$ lines 
($15 I_2$, $12 I_2$, $3 I_2 + 6 III$ and beyond)
would contradict this bound,
partly  due to the extra fibre of type $III$ given by $F_0=\ell_0+Q$.
\end{proof}

\subsection{Proof of Proposition \ref{prop:64} for cases (\ref{case1}) - (\ref{case3}) of Lemma \ref{lem:config}}
We shall use that $\pi_1$ always has a fibre of Kodaira type $I_3$ or $IV$
in these three cases.
In other words, there is a hyperplane $H$ splitting on $S$ as $\ell_1+\ell_2+\ell_3+\ell_4$.
Thus we can express the number of lines on $S$
in terms of the lines meeting either $\ell_i$ (as in the proof of Lemma \ref{lem:64}).
In order to reach more than 64 lines,
we have to compensate for $\ell_1$ only meeting 16 lines.
Since any line meets at most 20 lines on $S$ by Corollary \ref{cor:20},
this means that at least 2 out of $\ell_2, \ell_3, \ell_4$ meet 19 or 20 lines.
By Lemma \ref{lem:I}, these two lines are of the second kind.
But then, since the two lines are planar and thus intersect,
Lemma \ref{lem:I-II} implies that $S$ is the Schur quartic -- with 64 lines, contradiction. 
\qed

\subsection{Proof of Proposition \ref{prop:64} for case (\ref{case4}) of Lemma \ref{lem:config}}

Consider the 9 fibres $F_i $ of type $I_2$ (or possibly 3 among them of type $III$),
writing them as line plus conic
\[
F_i = \ell_i+Q_i \;\;\; (i = 2,\hdots,10).
\]
If all conics $Q_i$ were to be contained in the flecnodal divisor,
then the remaining part of degree 62 could account for at most 62 lines.
Hence we can assume that $Q_2$ (as well as $Q_3=Q_2^\sigma, Q_4=Q_2^{\sigma^2}$)
is not contained in the flecnodal divisor.
Thus $Q_2$ meets at most 36 lines on $S$ other than $\ell_1, \ell_2$.
This gives the following upper bound for the number of lines on $S$:
\[
\#\{\text{lines on }S\} \leq 10+36+\#\{\text{lines on $S$ meeting $\ell_2$}\}.
\]
Since $S$ was assumed to contain more than 64 lines, we infer that $\ell_2$ meets more than 18 lines.
Thus,  $\ell_2$ is of the second kind by Lemma \ref{lem:I}.
We also infer that the fibres $F_2, F_3, F_4$ of $\pi_1$ have Kodaira type $III$
 since $\ell_2$ has to meet $Q_2$ tangentially
 (cf.~the ramification analysis in the proof of Lemma \ref{lem:ram}).
 
 \begin{Lemma}
 The conics $Q_5,\hdots,Q_{10}$ are contained in the flecnodal divisor $\mF_S$.
 \end{Lemma}
 
 \begin{proof}
 Assume to the contrary that $Q_5$ is not contained in $\mF_S$ either.
 Then arguing as above, we find that $\ell_5$ is a line of the second kind, and $F_5$ has Kodaira type $III$. 
 Thus among our 10 reducible fibres of $\pi_1$ of Kodaira types $I_2$ and $III$, we have at least 5 $III$.
 This   causes $e(S)$ to exceed $24$, giving the required contradiction.
 \end{proof}

We continue by analysing the flecnodal divisor which we decompose as
\[
\mF_S = 65 \mbox{ lines} + Q_5 + \hdots + Q_{10} + \mR, \;\;\; \deg(\mR)=3.
\]
where the 65 lines comprise
$\ell_0, \ell_1$, the 18 components of $I_3, IV$ fibres of $\pi_0$, and 
the 45 lines including $\ell_2,\hdots,\ell_{10}$ which are sections of $\pi_0$.
By construction, both the sum over the 65 lines and $\mR$ are $\sigma$-invariant,
 and $\mR$ can a priori contain duplicates of some lines as well as the components
of a potential second singular fibre of $\pi_0$ of Kodaira type $I_2$.
Since $\ell_0$ is perpendicular to all the $Q_i(i>1)$ as components of different fibres of $\pi_1$,
we find using $\ell_0.\mF_S=20$
\begin{eqnarray*}
\label{eq:3}
3=\ell_0.\mR=
\deg(\mR).
\end{eqnarray*}
Immediately this implies that $ \mR$ consists of components of fibres of $\pi_0$.
If $ \mR$ were to be a full fibre of $\pi_0$, then $\ell_2.\mR=1$, and since $\ell_2.Q_i=0 \; (i>2)$,
we obtain the following contradiction by Corollary \ref{cor:20}:
\[
20 = \ell_2.\mF_S \leq -2+20+0+1=19.
\]
Next, if $ \mR$ were to split off an irreducible conic, 
then this would be a component of an $I_2$ fibre of $\pi_0$
which we computed to lie off $\mF_S$.
Finally assume that $ \mR$ decomposes into 3 lines.
Since $ \mR$ is $\sigma$-invariant, 
but cannot be a full fibre of $\pi_0$,
these lines can only be $\ell_1$ and possibly the linear component $\ell_1'$ of a second fibre of type $I_2$:
\[
\mR = a \ell_1 + (3-a) \ell_1',\;\;\; a\in\{0,1,2,3\}.
\]
Since $\ell_1.\mF_S=20$ implies $\ell_1.\mR=0$ and $\ell_1, \ell_1'$ are orthogonal as fibre components of $\pi_0$,
we infer that $a=0$ and $\pi_0$ has in fact a second $I_2$ fibre:
\[
\mR=3\ell_1'.
\]
But then we can argue for the fibration $\pi_1'$ induced by $\ell_1'$ as for $\pi_1$.
In consequence, exactly the same reasoning starting from \ref{ss:2nd} goes through to show
that Assumption \ref{ass} implies that $\ell_1'$ meets at most 10 lines on $S$
(case (\ref{case4}) of Lemma \ref{lem:config} applied to $\pi_1'$).
Hence we obtain the estimate
\[
20 = \ell_1'.\mF_S \leq 10 + 12 - 6 = 16.
\]
This gives the required contradiction to Assumption \ref{ass},
thus completing the proof of Proposition \ref{prop:64}.
\qed

\subsection{Proof of the main result}
Theorem \ref{thm}  follows from the combination of Proposition \ref{prop:64} and Lemma \ref{lem:64} (a).
\qed

\begin{Example}
\label{ex}
The subfamily of $\mZ$ of quartics such that $\pi_0$ has a fibre of Kodaira type $I_2$
does in fact admit surfaces with as many lines as 60.
For instance, consider the quartic defined by
\[
S=\{x_1^3 x_3+x_1 x_2 x_3^2+x_2^3 x_4+r x_3^3 x_4-x_1 x_2 x_4^2-r x_3 x_4^3=0\}\subset\PP^3
\] 
with $r=-16/27$.
Then this is smooth outside characteristics $2,3,5$ and contains 60 lines,
20 of which meet the line of the second kind given by $\{x_3=x_4=0\}$.
Over $\C$, we obtain a singular K3 surface with Picard number $\rho=20$ and discriminant $-15$ or $-60$.
\end{Example}

\begin{Remark}
Most arguments in this section remain valid in characteristic $p=2$ as well.
This is due to their geometric nature,
since there are no degenerations.
To see this, note first that the flecnodal divisor is non-degenerate since it does not contain the conic $Q$.
Furthermore, the fibrations $\pi_0, \pi_1$ cannot be quasi-elliptic 
since $\pi_0$ contains fibres of Kodaira type $I_2, I_3, IV$ not occurring on quasi-elliptic fibrations in characteristic $2$, and $\pi_1$ has the smooth fibre $E_0$ at one of the fixed points of $\sigma$ on $\ell_1$.
\end{Remark}

\section{Quartics with 64 lines}
\label{s:64}

To conclude this note,
we attempt to collect all results necessary to describe the line configuration on a smooth quartic with 64 lines.
Unfortunately, we need an assumption as pointed out by Tomasz Szemberg and Davide Veniani:

\begin{table}[ht!]
\begin{tabular}{cl}
\multirow{2}{*}{(*)} & 
\text{for any line $\ell\subset S$, if $\ell$ intersects two (coplanar) lines $\ell_1,\ell_2\subset S$ in their}\\
& \text{intersection point $P$,
then also the residual line $\ell_3$ goes through $P$.}
\end{tabular}
\end{table}

\begin{Proposition}
\label{prop:=64}
Let $S$ be a smooth quartic surface containing 64 lines.
Assume (*).
Then each line meets exactly 18 other lines on $S$;
these are arranged in a total number of 6 triangles (Kodaira type $I_3$) and stars
(Kodaira type $IV$).
\end{Proposition}

To sketch a proof of Proposition \ref{prop:=64},
we need a few preparations. 
First we note the following upper bound for incident lines:

\begin{Lemma}
\label{lem:<=18}
If the smooth quartic $S$ contains 64 lines,
then any line on $S$ meets at most $18$ other lines on $S$.
\end{Lemma}

\begin{proof}
For lines of the first kind, this is exactly Lemma \ref{lem:I}.
For lines of the second kind, only surfaces $S\in \mZ$ could possibly violate our claim;
this happens exactly when the fibration $\pi_0$ induced by a line $\ell_0$ of the second kind
has an additional singular fibre of type $I_2$,
consisting of a line $\ell_1$ and a residual conic $Q$
in the notation of the previous section.
Recall the automorphism $\sigma$ of order $3$ acting on $S$.
Since $\sigma$ permutes the components of the $I_3/IV$ fibres of $\pi_0$,
no section of $\pi_0$ can be fixed by $\sigma$.
Thus the sections of $\pi_0$ come in triples (resembling Lemma \ref{lem:triples}),
and in particular 
\[
\#\{\text{lines on $S$ not meeting }\ell_0\}\equiv 0\mod 3.
\]
For the total number of lines on $S$, we therefore find the congruence
\[
 \#\{\text{lines on $S$}\}\equiv 1 +  \#\{\text{lines on $S$ meeting }\ell_0\}  \mod 3.
\]
Note that modulo $3$ the number of lines meeting $\ell_0$ 
exactly equals the number of fibres of Kodaira type $I_2$ of $\pi_0$.
In particular, we find that the existence of a fibre of Kodaira type $I_2$ of $\pi_0$ implies the equivalence
\begin{eqnarray}
\label{eq:cong}
 \#\{\text{lines on $S$}\}> 64 \Longleftrightarrow  \#\{\text{lines on $S$}\} \geq 64.
 \end{eqnarray}
 Hence the arguments in Section \ref{s:II} show that 
 a surface $S\in\mZ$ may admit 64 lines only if $\pi_0$ does not have any fibre of Kodaira type $I_2$.
 That is, any line of the second kind contained in $S$ meets exactly 18 other lines on $S$.
\end{proof}

We shall now sketch 
how to prove the claim of Proposition \ref{prop:=64} concerning the arrangement of lines meeting a given one:

\begin{Lemma}
\label{lem:aux}
If a line
satisfying (*) on a smooth quartic   is met by at least 18 other lines,
then 18 of them are arranged in triangles and stars.
\end{Lemma}

\begin{proof}[Idea of proof]
For lines of the second kind, this is exactly Lemma \ref{lem:fibres} or \ref{lem:Z}.
For a line $\ell$ of the first kind, 
one can carry out a similar analysis of the ramification of the $3:1$ morphism $\ell\to\PP^1$
as in Sections~\ref{s:3}, \ref{s:4}.
In the same notation, the main difference is
that the three sections $O, P_1, P_2$ into which $\ell$ splits on the base change $S_2$
will usually not be 3-torsion.

However, if we assume that $\ell$ meets 18 other lines on $S$ or more,
then we can calculate the possible heights of the sections $P_1, P_2$
using the theory of Mordell-Weil lattices \cite{ShMW}.
Here the assumption (*) enters crucially as it ensures that the correction terms from the singular fibres
distribute evenly over the sections in a way compatible
with the intersection of the line $\ell$ on $S$ with the reducible residual cubics.
A detailed analysis 
reveals that for all heights to be non-negative both sections on $S_2$ indeed have to be torsion to start with;
but then for the fibres to accommodate two non-trivial torsion sections,
the fibre types have to include 6 times $I_3$ and $IV$ and the torsion order has to be $3$.
In particular, this confirms that $\ell$ meets 18 lines in 6 triangles or stars.
\end{proof}

\begin{Remark}
The above argument gives an alternative proof of Corollary \ref{cor:20}.
One could try to extend it for a geometric proof of Lemma \ref{lem:I}.
This would require ruling out that $\ell$ meets  additional lines outside the 6 triangles and stars
(contrary to what happens for lines of the second kind in Section~\ref{s:4}).
For instance, this can be achieved by a calculation with the Hessian as in the proof of Lemma \ref{lem:I}.
Since we are not aware
of a simple enough argument avoiding the use of the Hessian,
we omit the details here.
\end{Remark}

\begin{proof}[Proof of Proposition \ref{prop:=64}]
By Lemma \ref{lem:aux}, it remains to show
that any line meets exactly 18 others on $S$.
To see this, we use that $S$ contains 4 coplanar lines,
each meeting exactly 18 other lines on $S$.
Outside the family $\mZ$, this is exactly Lemma \ref{lem:64} (b);
within $\mZ$, we use the fibration induced by a line of the second kind and Lemma \ref{lem:<=18}.
By Lemma \ref{lem:aux}, each of these 4 coplanar lines induces an elliptic fibration
with 6 fibres of Kodaira type $I_3$ or $IV$.
The remaining $45$ lines form sections of the fibration,
15 of them meeting any given fibre component (by Lemma \ref{lem:<=18} again).

Now let $\ell$ be any line contained in $S$. Then $\ell$ meets one of the 4 coplanar lines.
Hence it is a fibre component for the induced fibration and therefore meets the following lines on $S$:
15 sections, the two other fibre components and the line inducing the fibration.
This amounts to the total number of 18 lines on $S$ met by $\ell$ as required.
\end{proof}

\begin{Remark}
It should be pointed out that any line on Schur's quartic \eqref{eq:Schur} is met by exactly 18 other lines
although the surface does not satisfy (*).
\end{Remark}

%
%
%
%
%
%
%
%

%
%


%

\subsection*{Acknowledgements}

We are indebted to Wolf Barth for sharing his insights on the subject
starting more than 10 years ago.
Thanks to Achill Sch\"urmann for helpful discussions on quadratic forms,
and to Tomasz Szemberg and Davide Veniani for pointing out an error in an earlier version
related to Proposition \ref{prop:=64}.
We are grateful to Igor Dolgachev, Duco van Straten and the anonymous referee 
for their valuable comments.
This project was started in March 2011 when MS enjoyed the hospitality of the Jagiellonian University in Krakow.
Special thanks to S\l awomir Cynk.

\end{document}